\documentclass[12pt]{article}

\usepackage{amsmath,amssymb,amsbsy,amsfonts,amsthm,latexsym,
      amsopn,amstext,amsxtra,amscd,stmaryrd,fullpage}

\usepackage[mathscr]{eucal}
\usepackage{comment}
\usepackage{color}
\usepackage{graphicx}
\begin{document}

\newtheorem{theorem}{Theorem}
\newtheorem{lemma}{Lemma}
\newtheorem{corollary}{Corollary}
\newtheorem{proposition}[theorem]{Proposition}

\theoremstyle{definition}
\newtheorem*{definition}{Definition}
\newtheorem*{remark}{Remark}
\newtheorem*{example}{Example}


\def\cA{\mathcal A}
\def\cB{\mathcal B}
\def\cC{\mathcal C}
\def\cD{\mathcal D}
\def\cE{\mathcal E}
\def\cF{\mathcal F}
\def\cG{\mathcal G}
\def\cH{\mathcal H}
\def\cI{\mathcal I}
\def\cJ{\mathcal J}
\def\cK{\mathcal K}
\def\cL{\mathcal L}
\def\cM{\mathcal M}
\def\cN{\mathcal N}
\def\cO{\mathcal O}
\def\cP{\mathcal P}
\def\cQ{\mathcal Q}
\def\cR{\mathcal R}
\def\cS{\mathcal S}
\def\cU{\mathcal U}
\def\cT{\mathcal T}
\def\cV{\mathcal V}
\def\cW{\mathcal W}
\def\cX{\mathcal X}
\def\cY{\mathcal Y}
\def\cZ{\mathcal Z}


\def\sA{\mathscr A}
\def\sB{\mathscr B}
\def\sC{\mathscr C}
\def\sD{\mathscr D}
\def\sE{\mathscr E}
\def\sF{\mathscr F}
\def\sG{\mathscr G}
\def\sH{\mathscr H}
\def\sI{\mathscr I}
\def\sJ{\mathscr J}
\def\sK{\mathscr K}
\def\sL{\mathscr L}
\def\sM{\mathscr M}
\def\sN{\mathscr N}
\def\sO{\mathscr O}
\def\sP{\mathscr P}
\def\sQ{\mathscr Q}
\def\sR{\mathscr R}
\def\sS{\mathscr S}
\def\sU{\mathscr U}
\def\sT{\mathscr T}
\def\sV{\mathscr V}
\def\sW{\mathscr W}
\def\sX{\mathscr X}
\def\sY{\mathscr Y}
\def\sZ{\mathscr Z}


\def\fA{\mathfrak A}
\def\fB{\mathfrak B}
\def\fC{\mathfrak C}
\def\fD{\mathfrak D}
\def\fE{\mathfrak E}
\def\fF{\mathfrak F}
\def\fG{\mathfrak G}
\def\fH{\mathfrak H}
\def\fI{\mathfrak I}
\def\fJ{\mathfrak J}
\def\fK{\mathfrak K}
\def\fL{\mathfrak L}
\def\fM{\mathfrak M}
\def\fN{\mathfrak N}
\def\fO{\mathfrak O}
\def\fP{\mathfrak P}
\def\fQ{\mathfrak Q}
\def\fR{\mathfrak R}
\def\fS{\mathfrak S}
\def\fU{\mathfrak U}
\def\fT{\mathfrak T}
\def\fV{\mathfrak V}
\def\fW{\mathfrak W}
\def\fX{\mathfrak X}
\def\fY{\mathfrak Y}
\def\fZ{\mathfrak Z}


\def\C{{\mathbb C}}
\def\F{{\mathbb F}}
\def\K{{\mathbb K}}
\def\L{{\mathbb L}}
\def\N{{\mathbb N}}
\def\Q{{\mathbb Q}}
\def\R{{\mathbb R}}
\def\Z{{\mathbb Z}}
\def\E{{\mathbb E}}
\def\T{{\mathbb T}}
\def\P{{\mathbb P}}
\def\D{{\mathbb D}}


\def\eps{\varepsilon}
\def\mand{\qquad\mbox{and}\qquad}
\def\\{\cr}
\def\({\left(}
\def\){\right)}
\def\[{\left[}
\def\]{\right]}
\def\<{\langle}
\def\>{\rangle}
\def\fl#1{\left\lfloor#1\right\rfloor}
\def\rf#1{\left\lceil#1\right\rceil}
\def\le{\leqslant}
\def\ge{\geqslant}
\def\ds{\displaystyle}

\def\xxx{\vskip5pt\hrule\vskip5pt}
\def\yyy{\vskip5pt\hrule\vskip2pt\hrule\vskip5pt}
\def\imhere{ \xxx\centerline{\sc I'm here}\xxx }

\newcommand{\comm}[1]{\marginpar{
\vskip-\baselineskip \raggedright\footnotesize
\itshape\hrule\smallskip#1\par\smallskip\hrule}}


\def\e{\mathbf{e}}
\def\sPrc{{\displaystyle \sP_r^{(c)}}}

\title{\bf Zeros of Complex Random Polynomials Spanned by Bergman Polynomials
}

\author{
\sc Marianela Landi, Kayla Johnson, Garrett Moseley,\\ and \sc Aaron Yeager }

\date{}
\newcommand{\Addresses}{{
  \bigskip
  \footnotesize

  Marianela Landi,
  \textit{E-mail address}: \texttt{920131101@student.ccga.edu}
  \vspace{0.1in}

  Kayla Johnson, 
  \textit{E-mail address}: \texttt{920117897@student.ccga.edu}
  \vspace{0.1in}

  Garrett Moseley, 
  \textit{E-mail address}: \texttt{920112221@student.ccga.edu}
\vspace{0.1in}

  Aaron Yeager, 
  \textit{E-mail address}: \texttt{ayeager@ccga.edu}
\vspace{0.1in}

\textsc{Department of Mathematics, College of Coastal Georgia,
    Brunswick, Georgia 31520}
}}

\maketitle

\begin{abstract}
We study the expected number of zeros of
$$P_n(z)=\sum_{k=0}^n\eta_kp_k(z),$$
where $\{\eta_k\}$ are complex-valued i.i.d standard Gaussian random variables, and $\{p_k(z)\}$ are polynomials orthogonal on the unit disk.  When $p_k(z)=\sqrt{(k+1)/\pi} z^k$, $k\in \{0,1,\dots, n\}$, we give an explicit formula for the expected number of zeros of $P_n(z)$ in a disk of radius $r\in (0,1)$ centered at the origin.  From our formula we establish the limiting value of the expected number of zeros, the expected number of zeros in a radially expanding disk, and show that the expected number of zeros in the unit disk is $2n/3$.  Generalizing our basis functions $\{p_k(z)\}$ to be regular in the sense of Ullman--Stahl--Totik, and that the measure of orthogonality associated to polynomials is absolutely continuous with respect to planar Lebesgue measure, we give the limiting value of the expected number of zeros of $P_n(z)$ in a disk of radius $r\in (0,1)$ centered at the origin, and show that asymptotically the expected number of zeros in the unit disk is $2n/3$.

\end{abstract}

\textbf{2020 Mathematics Subject Classification :} 30C15, 30E15, 30C40, 60B99.

\textbf{Keywords:} Random Polynomials, Bergman Polynomials, Ullman--Stahl--Totik Regular.

\section{Background}

Most students of mathematics are familiar with the idea of the roots of a polynomial and have found the roots of polynomials of low degree by hand.  With the aid of numerical algorithms and modern computers, it is an easy task to graph the roots of a polynomial of relatively large degree in the complex plane and to examine their distribution.  A central question in the field of random polynomials is how these roots are distributed on average when the polynomial is chosen at random.  For example, if a polynomial is chosen at random, how many real roots should we expect to have? These types of problems go back to pioneering work in the 1930's by Bloch and P\'olya \cite{BP}, and Littlewood and Offord  \cite{LO3}.  Applications of random polynomials are very abundant, arising in perturbation theory, the study of difference and differential equations, random matrix theory, the study of approximate solution of operator equations, the method of least squares estimates, economics, statistical communication theory, and mathematical physics.

In order to make the questions raised in the previous paragraph precise, it is necessary to decide what a random polynomial is.  A \emph{random polynomial} is a polynomial
$$p_n(z)=\sum_{k=0}^n\eta_kz^k,$$
where $\{\eta_k\}$ are random variables.  Let $N_n(S)$ denote the number of zeros of $p_n(z)$ in a set $S$.  Note that when we take $S$ to be the entire complex plane ($\C$), as the degree of $p_n(z)$ is $n$, by the Fundamental Theorem of Algebra we have $N_n(\C)=n$.  For $S\subsetneq\C$, we would like to know $N_n(S)$.  For example, we could take $S$ to be the real line ($\R$), $[-1,1]$, the unit circle ($\T$), the unit disk ($\D$), or $[-1,1]\times [-1,1]$, etc.  As the coefficients $\{\eta_k\}$ are random variables, this involves the computation of $\E[N_n(S)]$, where $\E$ denotes the expectation.

In 1943, Kac \cite{K1} produced an integral equation for the expected number of real zeros of $p_n(z)$ when the random variables $\{\eta_j\}$ are real-valued independent and identically distributed (i.i.d.) standard Gaussian.  We note that independently in 1945, while studying random noise Rice \cite{R} derived a similar formula for $\E[N_n(\R)]$ in the Gaussian setting.  After Kac established the formula for the expected number of zeros, he proved that the the number of real zeros of the random polynomial $p_n(z)$ is asymptotic to $(2/\pi)\log n$ as $n\rightarrow \infty$.

Due to the work of Kac and Rice, formulas for the density function for the expected number of zeros of a random polynomial, called the \emph{intensity function},
are known as \emph{Kac-Rice formulas}.  Thus the intensity function, which we denote as $\rho_n(z)$, is the function that satisfies
$$\E[N_n(\Omega)]=\int_{\Omega} \rho_n(z) \ dz, \ \ \ \ \text{where} \ \ \Omega \subset \C.$$

We now consider the random polynomials of the form
$$f_n(z)=\sum_{k=0}^n\eta_k z^k,$$
 where $\{\eta_k\}$ are i.i.d.~complex-valued standard Gaussian random variables.  That is, when $\eta_j=\alpha_j+i\beta_j$, where $\alpha_j$ and $\beta_j$ are i.i.d.~real-valued standard Gaussian for all $j\in \{0,1,\dots,n\}$.
The classic result of Hammersley \cite{HM}  says that the intensity function for a complex random polynomial $f_n(z)$ is given by
\begin{equation*}
\rho_n(z)=\frac{1}{\pi}\frac{1-|h_{n+1}(z)|^2}{(1-|z|^2)^2},\quad \text{where} \quad h_{n+1}(z)=\frac{(1-|z|^2)(n+1)z^n}{1-|z|^{2(n+1)}}.
\end{equation*}

As noted by Arnold \cite{Arn}, for a disk of radius $r$, denoted as $D(0,r)$, it follows that 
\begin{equation*}
\E[N_n(D(0,r))]=
\begin{cases}
\displaystyle\frac{r^2}{1-r^2}-(n+1)\frac{r^{2n+2}}{1-r^{2n+2}}, \ \ &r\in(0,1),\\[3ex]
\displaystyle\frac{n}{2}, \ \ &r=1,
\end{cases}
\end{equation*}
and
\begin{equation*}
\lim_{n\rightarrow \infty}\E[N_n(D(0,r))]=\frac{r^2}{1-r^2}, \quad \in(0,1).
\end{equation*}
When examining an expanding disk centered at the origin, Ledoan et.~al.~\cite{AL2}  proved the following scaling limit
\begin{equation*}
\lim_{n\rightarrow \infty}\frac{\E[N_n(D(0,e^{-t/2n}))]}{n}=\frac{1}{t} + \frac{1}{1-e^t}, \quad t>0.
\end{equation*}

For other early results in random polynomials, we refer the reader to the books by
 Bharucha-Reid and Sambandham \cite{BR} and Farahmand \cite{Fa}.

\section{Main Results}

In our work we study the expected number of zeros of random polynomials of the form
\begin{equation}\label{ROP2}
P_n(z)=\sum_{k=0}^n \eta_k p_k(z),
\end{equation}
where $\{\eta_k\}$ are i.i.d complex-valued standard Gaussian, and $\{p_k(z)\}$ are polynomials orthogonal on $\D$ with respect to finite positive Borel measure $\mu$.  By $\{p_k(z)\}$ being orthogonal on $\D$, we mean that they satisfy the relation
\begin{equation}\label{orthoprop}
\int_{\D}p_n(z)\overline{p_m(z)} \ d\mu(z)=
\begin{cases}
1, \ &  \text{when} \ n=m,\\
0, \  &  \text{when} \ n\neq m.
\end{cases}
\end{equation}
In this case, the spanning functions $\{p_k(z)\}$ are called \emph{Bergman polynomials}.   A random polynomial that is spanned by polynomials that satisfy an orthogonality relation is said to be a \emph{Random Orthogonal Polynomial}.

We will  assume that the measure of orthogonality $\mu$ is absolutely continuous with respect to planar Lebesgue measure $dA$, and our spanning orthogonal polynomials  
\begin{equation*}
p_k(z)=\kappa_k z^k+ \text{lower terms}, \  \ k\in\{0,1,\dots,n\},
\end{equation*}
satisfy $\kappa_k>0$ with 
\begin{equation}\label{SUT}
\lim_{k\rightarrow \infty}\kappa_k^{1/k}=1.
\end{equation}
Orthogonal polynomials on the unit disk which satisfy  \eqref{SUT} are called \emph{regular in the sense of Ullman--Stahl--Totik}.
It is known that such Bergman polynomials are used as a basis for representing
analytic functions in the unit disk (c.f. Stahl and Totik \cite{ST}).

As an example of such spanning orthogonal polynomials, consider  $p_k(z)=\sqrt{(k+1)/\pi}z^k$, $k\in \{0,1,\dots , n\}$, with weight function $h(z)=1$.  It is clear that these polynomials satisfy \eqref{SUT} and   
\begin{align*}
\int_\D \sqrt{\frac{n+1}{\pi}}z^n\overline{\sqrt{\frac{m+1}{\pi}}z^m}\ dA(z)&=\int_0^1\int_0^{2\pi} \sqrt{\frac{n+1}{\pi}}(re^{i\theta})^n\overline{\sqrt{\frac{m+1}{\pi}}(re^{i\theta})^m}\ r \ d\theta \ dr \\[2ex]
&=\frac{\sqrt{(n+1)(m+1)}}{\pi}\int_0^1\int_0^{2\pi}r^{n+m+1}e^{i\theta(n-m)} \ d\theta \ dr\\[2ex]
&=\begin{cases}
1, \ & \text{when} \ n=m,\\
0, \ & \text{when} \ n\neq m .
\end{cases}
\end{align*}

For other examples of  Bergman polynomials, we leave it to the reader to verify that
$$p_k(z)=\sqrt{ \frac{(k+1)(k+j+1) }{\pi j} }z^k, \ \ \text{where} \ j>0, \ \text{with} \ h(z)=1-|z|^{2j},$$
and 
$$p_k(z)=\frac{2}{\sqrt{\pi (k+1)(k+2)(k+3)}}\sum_{j=0}^k(j+1)z^j(1+z+\cdots + z^{k-j}), \ \ \text{with} \ h(z)=|z-1|^2,$$
also satisfy \eqref{orthoprop} and \eqref{SUT}.

Explicit formulas and asymptotics for the expected number of zeros of random polynomials spanned by polynomials orthogonal on the real-line or spanned by orthogonal polynomials the unit circle have been given much attention.  In contrast, such results for random (planar) orthogonal polynomials have not been as widely studied, and thus provides motivation for our investigation.  We remark that our study is also motivated by applications in mathematical physics, where random polynomials spanned by Bergman polynomials on the unit disk are the logarithmic derivative of the characteristic function of a random $n\times n$ unitary matrix (c.f. Diaconis and Evans \cite{DE}).

Our main tool in examining the expected number of zeros for $P_n(z)$ is a formula for the intensity function given by Yeager \cite{AY3} (and independently by Ledoan \cite{AL}).  The formula for the intensity function, which we state in the next section, allows us to obtain the following:

\begin{theorem}\label{Thm1}
Let $P_n(z)=\sum_{k=0}^n \eta_k \sqrt{\frac{k+1}{\pi}}z^k$, where $\{\eta_k\}$ are complex-valued i.i.d.~standard Gaussian.  The following formulas hold valid:
\begin{enumerate}
\item[\textup{(i)}] The intensity function can be written as
\begin{equation}\label{intbf}
\rho_n(z)=\frac{1}{\pi}\left( \frac{2}{(1-|z|^2)^2}-\frac{(n+1)(n+2)|z|^{2n}(|z|^{2n+4}-(n+2)|z|^2+n+1)}{(1+|z|^{2n+2}((n+1)|z|^2-(n+2)))^2} \right).
\end{equation}

\item[\textup{(ii)}] For $r\in (0,1)$, we have 
\begin{equation}\label{zerosindiskf}
\E[N_n(D(0,r))]=
\displaystyle \frac{2r^2}{1-r^2}-\frac{(n+1)(n+2)(1-r^2)r^{2n+2}}{1+r^{2n+2}((n+1)r^2-(n+2))}.
\end{equation}

\item[\textup{(iii)}] When $r\in (0,1)$, it follows that
\begin{equation}\label{limindisk}
\lim_{n\rightarrow \infty}\E[N_n(D(0,r))]=\frac{2r^2}{1-r^2}.
\end{equation}

\item[\textup{(iv)}] The expected number of zeros in the expanding disk $D(0,e^{-t/2n})$ possesses the property that
\begin{equation}\label{expdiskf}
\lim_{n\rightarrow \infty}\frac{\E[N_n(D(0,e^{-t/2n}))]}{n}=\frac{2}{t}+\frac{t}{1-e^{t}+t}, \quad t>0.
\end{equation}

\item[\textup{(v)}] Over the whole unit disk we have 
\begin{equation}\label{zerosinunitdisk}
\E[N_n(\D)]=\frac{2n}{3}.
\end{equation}
\end{enumerate}  
\end{theorem}

Generalizing our spanning orthogonal polynomials we are able to retain \eqref{limindisk} and establish \eqref{zerosinunitdisk} in an asymptotic sense.
\begin{theorem}\label{Thm2}
Let $P_n(z)=\sum_{k=0}^n \eta_k p_k(z)$, where $\{\eta_k\}$ are complex-valued i.i.d.~standard Gaussian, and $\{p_k(z)\}$ are Bergman polynomials that satisfy \eqref{SUT}  and are such that their corresponding measure of orthogonality $\mu$ is absolutely continuous with respect to planar Lebesgue measure $dA$. 
\begin{enumerate}
\item[\textup{(i)}] When $r\in (0,1)$, we have 
\begin{equation}\label{limindisk2}
\lim_{n\rightarrow \infty}\E[N_n(D(0,r))]=\frac{2r^2}{1-r^2}.
\end{equation}

\item[\textup{(ii)}] In the whole unit disk it follows that
\begin{equation}\label{asydisk}
\lim_{n\rightarrow \infty}\frac{\E[N_n(\D)]}{n}=\frac{2}{3}.
\end{equation}
\end{enumerate}
\end{theorem}

\subsection{Concluding Remarks and Future Research}

Due to Corollary 3.2 given by Pritsker and Yeager in \cite{PritYgr15}, it is known that in the setting of our theorems, the zeros of $P_n(z)$ are accumulating near the unit circle as $n\rightarrow \infty$.  In light of \eqref{zerosinunitdisk} and \eqref{asydisk}, we now know that two thirds of these zeros are inside the unit disk. We conjecture that this phenomenon is occurring because the spanning functions $\{p_k(z)\}$ are orthogonal on the unit disk, and that this orthogonality is so strong it is ``grabbing" zeros.  Below we give a visual presentation of the zeros in $[-1.5,1.5]\times [-1.5,1.5]$ for 4000 different random polynomials  each of degree $25$  for the case when our spanning functions are $p_k(z)=\sqrt{(k+1)/\pi}z^k$, $k\in \{0,1,\dots, 25\}$.

\begin{figure}[t]
\centerline{\includegraphics[width=3in, height=3in]{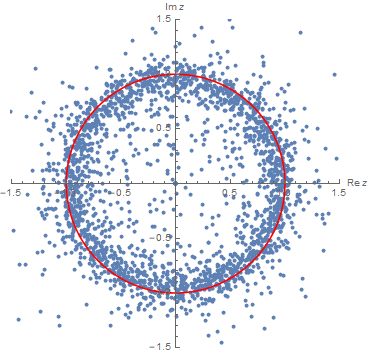}}
\caption{In red is the unit circle and in blue are the zeros in $[-1.5,15]\times [-1.5,1.5]$ for 4000 different random polynomials of the form  $P_{25}(z)=\sum_{k=0}^{25}\eta_k\sqrt{(k+1)/\pi }z^k$.}
\label{fig}
\end{figure}

\section{The Proofs}

Let $\{f_j\}$ be a sequence of polynomials such that $\deg f_j(z)=j$, for $j\in \{0,1,\dots, n\}$.  Set
\begin{equation*}
\label{PRV}
P_n(z)=\sum_{j=0}^n \eta_j f_j(z), 
\end{equation*}
where $n$ is a fixed integer, and $\{\eta_j\}$ are complex-valued i.i.d.~standard Gaussian.
The expected number of zeros of $P_n(z)$ will be given in terms of the kernels 
\begin{equation*}
K_{n}(z,z)=\sum_{j=0}^{n}f_j(z)\overline{f_j(z)},\ \ \ \ \ \ \ K_{n}^{(0,1)}(z,z)=\sum_{j=0}^{n} f_j(z)\overline{f_j^{\prime}(z)}, \ \ \ \ \ \ \ K_{n}^{(1,1)}(z,z)=\sum_{j=0}^{n}f_j^{\prime}(z)\overline{f_j^{\prime}(z)}.
\end{equation*}

For $\Omega \subset \C$ a Jordan region, in \cite{AY3} (c.f. pp. 119, 120 and 134) it was shown that 
\begin{align}
\label{intbdy}
\E[N_n(\Omega)]&=\frac{1}{2\pi i} \int_{\partial \Omega}\frac{\overline{ K_n^{(0,1)}(z,z)}}{K_n(z,z)} \ dz \\[2ex]
\nonumber
&=\int_{\Omega}\rho_n(x,y)\ dx \ dy ,\ \ \ z=x+iy,
\end{align}
where 
\begin{equation}
\label{intthm2.1RV}\rho_n(x,y)=\rho_n(z)=\frac{K_{n}^{(1,1)}(z,z)K_{n}(z,z)-\left|K_{n}^{(0,1)}(z,z)\right|^2}{\pi \left(K_{n}(z,z)\right)^2}.
\end{equation}
We remark that the intensity function $\rho_n(z)$ has no mass on the real-line.  Consequently, $\E[N_n(\R)]=0$.

Formula \eqref{intthm2.1RV} will be used in the proof of \eqref{intbf} from Theorem \ref{Thm1}.  Using \eqref{intbf}, we derive the results of \eqref{zerosindiskf}--\eqref{zerosinunitdisk}.   In the proof of Theorem \ref{Thm2}, the form of the intensity function given in \eqref{intthm2.1RV} will be applied to establish \eqref{limindisk2}, and that of \eqref{intbdy} will be used to prove \eqref{asydisk}.

\begin{proof}[Proof of Theorem \ref{Thm1}]

To establish \eqref{intbf}, observe that 
since
$$f_j(z)=\sqrt{ \frac{j+1}{\pi} }z^j, \ \ \text{for} \ \ j\in\{0,1,\dots,n\},$$
the above kernels simplify as
\begin{equation*}\label{K0}
K_n(z,z)=\sum_{k=0}^n\frac{k+1}{\pi}|z|^{2k}=\frac{1+|z|^{2(n+1)}((n+1)|z|^2-(n+2))}{\pi(1-|z|^2)^2},
\end{equation*}
\begin{equation*}\label{K01}
K_n^{(0,1)}(z,z)=\sum_{k=0}^n\frac{k(k+1)}{\pi}z^k\overline{z}^{k-1}=\frac{2zK_n(z,z)}{1-|z|^2}-\frac{(n+1)(n+2)z^{n+1}\overline{z}^n}{\pi(1-|z|^2)},
\end{equation*}
and
\begin{align*}
K_n^{(1,1)}(z,z)&=\sum_{k=0}^n\frac{k^2(k+1)}{\pi}|z|^{2(k-1)}\\[2ex]
&=\frac{2(1+2|z|^2)K_n(z,z)}{(1-|z|^2)^2}-\frac{(n+1)(n+2)|z|^{2n}(1+2|z|^2)}{\pi(1-|z|^2)^2}-\frac{n(n+1)(n+2)|z|^{2n}}{\pi(1-|z|^2)}.
\end{align*}
After much algebraic simplification we see that
\begin{align*}
K_n(z,z)K_n^{(1,1)}(z,z)-|K_n^{(0,1)}(z,z)|^2&=\frac{2K_n(z,z)^2}{(1-|z|^2)^2}-\frac{n(n+1)(n+2)|z|^{2n}K_n(z,z)}{\pi(1-|z|^2)}\\[2ex]
&\quad -\frac{(n+1)(n+2)|z|^{2n}(1-2|z|^2)K_n(z,z)}{\pi(1-|z|^2)^2}\\[2ex]
&\quad -\frac{(n+1)^2(n+2)^2|z|^{4n+2}}{\pi^2(1-|z|^2)^2}.
\end{align*}
Thus, using  \eqref{intthm2.1RV} and further simplifying we achieve
\begin{align*}
\rho_n(z)&=\frac{K_n(z,z)K_n^{(1,1)}(z,z)-|K_n^{(0,1)}(z,z)|^2}{\pi K_n(z,z)^2}\\[2ex]
&=\frac{1}{\pi}\left(\frac{2}{(1-|z|^2)^2}-\frac{(n+1)(n+2)|z|^{2n}(|z|^{2n+4}-(n+2)|z|^2+n+1)}{(1+|z|^{2n+2}((n+1)|z|^2-(n+2)))^2}\right),
\end{align*}
and thus completes the justification of \eqref{intbf}.

We will now establish the result given in \eqref{zerosindiskf}.  To this end, for $0<r<1$, using the above form of the intensity function and then changing to polar coordinates yields
\begin{align}
\nonumber
\E[N_n(D(0,r))]&=\int_{D(0,r)}\rho_n(z)\ dA(z)\\[2ex]
\nonumber
&=\frac{1}{\pi}\int_0^{2\pi}\int_0^r 
\left( \frac{2}{(1-t^2)^2}-\frac{(n+1)(n+2)t^{2n}(t^{2n+4}-(n+2)t^2+n+1)}{(1+t^{2n+2}((n+1)t^2-(n+2)))^2} \right) t  dt  d\theta\\[2ex]
\nonumber
&=\int_0^r \left( \frac{4t}{(1-t^2)^2}-\frac{2(n+1)(n+2)t^{2n+1}(t^{2n+4}-(n+2)t^2+n+1)}{(1+t^{2n+2}((n+1)t^2-(n+2)))^2} \right)  dt.
\end{align}

Observe that 
\begin{equation*}
F(t):=\frac{2}{1-t^2}-\frac{(n+2)(1-t^{2n+2})}{1+t^{2n+2}((n+1)t^2-(n+2))}
\end{equation*}
satisfies
\begin{equation*}
\frac{d}{dt}F(t)=\frac{4t}{(1-t^2)^2}-\frac{2(n+1)(n+2)t^{2n+1}(t^{2n+4}-(n+2)t^2+n+1)}{(1+t^{2n+2}((n+1)t^2-(n+2)))^2} .
\end{equation*}
Therefore, by the Fundamental Theorem of Calculus it follows that
\begin{align}
\nonumber
\E[N_n(D(0,r))]
&=\frac{2}{1-t^2}-\frac{(n+2)(1-t^{2n+2})}{1+t^{2n+2}((n+1)t^2-(n+2))}\ \Bigg|^{t=r}_{t=0}\\[2ex]
\label{zerosindiskf2}
&=\frac{2r^2}{1-r^2}-
\frac{(n+1)(n+2)(1-r^2)r^{2n+2}}{1+r^{2n+2}((n+1)r^2-(n+2))},
\end{align}
and we achieve the desired conclusion of \eqref{zerosindiskf}.

As $r\in (0,1)$ in the previous formula for the expected number of zeros in $D(0,r)$, we achieve 
\begin{equation*}
\lim_{n\rightarrow \infty}\E[N_n(D(0,r))]=\frac{2r^2}{1-r^2},
\end{equation*}
which establishes \eqref{limindisk}.

Now, observe that for $t>0$, if we set $r=  e^{-t/2n}$ in \eqref{zerosindiskf2} and scale by $n$, we have 
\begin{align*}
\lim_{n\rightarrow \infty}\frac{\E[N_n(D(0,e^{-t/2n}))]}{n}&=\lim_{n\rightarrow \infty}\left[\frac{2e^{-t/n}}{n(1-e^{-t/n})}-
\frac{(n+1)(n+2)(1-e^{-t/n})e^{-t(n+1)/n}}{n(1+e^{-t(n+1)/n}((n+1)e^{-t/n}-(n+2)))}\right]\\[2ex]
&=\frac{2}{t}+\frac{t}{1-e^{t}+t},
\end{align*}
thus showing the validity of \eqref{expdiskf}.

To prove  \eqref{zerosinunitdisk} we would like to set $r=1$ in \eqref{zerosindiskf2}, which as it is written, appears to have a singularity at $r=1$.  However, observe that \eqref{zerosindiskf2} can be written as 
\begin{equation*}
\E[N_n(D(0,r))]=\frac{2r^2}{1-r^2}-
\frac{(n+1)(n+2)(1-r^2)r^{2n+2}}{1+r^{2n+2}((n+1)r^2-(n+2))}=\displaystyle\frac{\displaystyle\sum_{k=1}^{n}k(k+1) r^{2k}}{1+\displaystyle\sum_{k=1}^{n}{(k+1)r^{2k}}}.
\end{equation*}

Therefore, taking $r=1$ in the above we achieve 
\begin{align*}
\E[N_n(\D)]=\displaystyle\frac{\displaystyle\sum_{k=0}^{n}k(k+1) }{\displaystyle\sum_{k=0}^{n}{(k+1)}}=\frac{\left (\displaystyle\frac{n(n+1)(n+2)}{3}\right)}{\left(\displaystyle \frac{(n+1)(n+2)}{2}\right)}=\frac{2n}{3},
\end{align*}
and hence completes the proof of \eqref{zerosinunitdisk}.

\end{proof}

\begin{proof}[Proof of Theorem \ref{Thm2}]

To establish \eqref{limindisk2}, we first note that it is well known that for $z\in \D$,  uniformly on compact subsets of the unit disk it follows that
\begin{align*}
\lim_{n\rightarrow \infty}K_n(z,z)=\lim_{n\rightarrow \infty}\sum_{k=0}^np_k(z)\overline{p_k(z)}=\frac{1}{\pi(1-|z|^2)^2}.
\end{align*} 
Thus taking the respective derivatives we achieve
\begin{equation}
\lim_{n\rightarrow \infty}K_n^{(0,1)}(z,z)= \frac{2z}{\pi(1-|z|^2)^3},\ \textup{and} \ \
\lim_{n\rightarrow \infty}K_n^{(1,1)}(z,z)= \frac{2+4|z|^2}{\pi(1-|z|^2)^4}.
\end{equation}
For $z\in \D$, using the above uniform limits along with the representation of the intensity function given in \eqref{intthm2.1RV} yields
\begin{align*}
\lim_{n\rightarrow \infty}\rho_n(z)=\lim_{n\rightarrow \infty}\frac{K_{n}^{(1,1)}(z,z)K_{n}(z,z)-\left|K_{n}^{(0,1)}(z,z)\right|^2}{\pi \left(K_{n}(z,z)\right)^2}
=\frac{2}{\pi (1-|z|^2)^2}.
\end{align*}

As the above limit holds uniformly on compact subsets of the unit disk, when $r\in (0,1)$, we can pass the limit through the integration to give
\begin{align*}
\lim_{n\rightarrow \infty}\E[N_n(D(0,r))]&=\int_{D(0,r)} \lim_{n\rightarrow \infty}\rho_n(z) \ dA(z)\\[2ex]
&=\int_{D(0,r)}\frac{2}{\pi (1-|z|^2)^2} \ dA(z) \\[2ex]
&=\int_0^{2 \pi} \int_0^r \frac{2}{\pi (1-t^2)^2} \ t \ dt \ d\theta \\[2ex]
&=\frac{2r^2}{1-r^2},
\end{align*}
and hence completes the desired result in \eqref{limindisk2}.

Under the hypothesis of Theorem \ref{Thm2}, due to the work of Lubinsky (c.f. Corollary 1.4 in \cite{Lub1}) we know that  uniformly for $z$ in an open arc of $\T$ we have 
\begin{equation}
\lim_{n\rightarrow \infty}\frac{\overline{K_n^{(0,1)}(z,z)}}{n \ K_n(z,z)}=\frac{2 }{3}\overline{z}.
\end{equation}
Taking the arcs $(0, \pi)$ and $(\pi,2 \pi)$, and as the intensity function has no mass on the real-line, using  \eqref{intbdy} with the above uniform limit we  achieve
\begin{align}
\nonumber
\lim_{n\rightarrow \infty}\frac{\E[N_n(\D)]}{n}&=\lim_{n\rightarrow \infty}\frac{1}{n} \frac{1}{2\pi i} \int_{\T}\frac{\overline{ K_n^{(0,1)}(z,z)}}{K_n(z,z)} \ dz \\[2ex]
\nonumber
&=\frac{1}{2\pi i} \int_{\T}\lim_{n\rightarrow \infty}\frac{\overline{ K_n^{(0,1)}(z,z)}}{n \ K_n(z,z)} \ dz \\[2ex]
\nonumber
&=\frac{1}{2\pi i} \int_{\T} \frac{2 }{3}\overline{z} \ dz \\[2ex]
\label{GT}
&=\frac{2}{3} \frac{1}{\pi}\int_{\D} \frac{d}{d\overline{z}} \overline{z} \ dA(z) \\[2ex]
\nonumber
&=\frac{2}{3} \frac{1}{\pi} \int_{\D} 1 \ dA(z) \\[2ex]
\nonumber
&=\frac{2}{3},
\end{align}
where we have used the complex version of Green's Theorem for \eqref{GT}.  Therefore we have established \eqref{asydisk}, and thus completed the proof of Theorem \ref{Thm2}.

\end{proof}

\Addresses


\begin{thebibliography}{bg}


\bibitem{Arn}L. Arnold, \"Uber die Nullstellenverteilung zuf\"alliger Polynome, Math. Z.  92 (1966), 12--18.





\bibitem{BP}
A. Bloch and G. P\'{o}lya, On the roots of a certain algebraic equation, Proc. Lond. Math. Soc. 33 (1932), 102--114.


\bibitem{BR}
A. T. Bharucha-Reid and M. Sambandham, Random polynomials, Academic Press, Orlando, 1986.



\bibitem{DE}
P. Diaconis and S. N. Evans, Linear functionals of eigenvalues of random matrices,
Trans. Amer. Math. Soc. (2001), 353(7), 2615–2633.


\bibitem{Fa}K. Farahmand, Topics in random polynomials, Pitman Res. Notes Math.  393, 1998.



\bibitem{AL2}
K. Ferrier, M. Jackson, A. Ledoan, D. Patel, H. Tran, The expected number of complex zeros of complex random polynomials,  Illinois J. Math.
    Vol. 61, Number 1--2 (2017), 211-224.




\bibitem{HM}
J. Hammersley, The zeros of a random polynomial, Proc. of the Third Berk. Sym. on Math. Stat. and Prob. 1954-1955 vol. II, University of Cal. Press, Berkeley and Los Angeles (1956) 89--111.



\bibitem{K1}
M.  Kac,  On  the  average  number  of  real  roots  of  a  random  algebraic  equation,  Bull.
Amer. Math. Soc. 49 (1943), 314--320.


\bibitem{AL}
A. Ledoan, Explicit formulas for the distribution of complex zeros of a family of random sums, J. Math. Ann. App. (2016) 444 (2), 1304--1320.







\bibitem{LO3}
J. Littlewood and A. Offord, On the number of real roots of a random algebraic equation, J. Lond. Math. Soc. 13 (1938), 288--295.




\bibitem{Lub1}
D. Lubinsky, Universality type limits for Bergman orthogonal polynomials,  Comput. Meth. Funct. Th. 10 (2010), 135--154.


\bibitem{PritYgr15}
I. Pritsker and A. Yeager, Zeros of polynomials with random coefficients, J. Approx. Theory 189 (2015), 88--100.


\bibitem{R}
S. Rice, Mathematical theory of random noise, Bell System Tech J. 25 (1945), 46--156.





\bibitem{ST} H. Stahl and V. Totik, General orthogonal polynomials, Cambridge Univ. Press, New York, 1992.



\bibitem{AY3}
A. Yeager, Random orthogonal polynomials, Ph.D. Dissertation, Oklahoma State University, 2019.


\end{thebibliography}
\end{document}